\documentclass{article}

\usepackage{amsmath, amsthm, amsfonts, amssymb, enumerate}

\newtheoremstyle{problemstyle}  % <name> This is my problemstyle. use begin{problem}.
        {12pt}% <space above>
        {}    % <space below>
        {}    % <body font>
        {}    % <indent amount}
        {\bfseries} % <theorem head font>
        {\normalfont\bfseries.} % <punctuation after theorem head>
        {.5em}% <space after theorem head>
        {}    % <theorem head spec (can be left empty, meaning `normal')>
        
\theoremstyle{problemstyle}

\newtheorem{problem}{Problem}[section]
\newtheorem{definition}[problem]{Definition}
\newtheorem{lemma}[problem]{Lemma}
\newtheorem{proposition}[problem]{Proposition}
\newtheorem{theorem}[problem]{Theorem}
\newtheorem{remark}[problem]{Remark}

\newtheorem{fact}[problem]{Fact}

\DeclareMathOperator{\defin}{def} % for definable Galois cohomology
\DeclareMathOperator{\tp}{tp} % types
\DeclareMathOperator{\aut}{Aut} % Automorphism
\DeclareMathOperator{\HH}{H} % cohomology
\DeclareMathOperator{\ZZ}{Z} % cocycles
\DeclareMathOperator{\PP}{P} % cocycles
\DeclareMathOperator{\acl}{acl} % acl
\DeclareMathOperator{\dcl}{dcl} % dcl

\newcommand{\xx}{\bar{x}}
\newcommand{\yy}{\bar{y}}

\newcommand{\aaa}{\bar{a}}
\newcommand{\bb}{\bar{b}}

\newcommand{\dd}{\bar{d}}
\newcommand{\ee}{\bar{e}}

\title{The short exact sequence in definable Galois cohomology}

\author{David Meretzky}

\date{\today}

\begin{document}

\maketitle

\begin{abstract}
    \noindent In \cite{pillay_1997}, Pillay introduced definable Galois cohomology, a model-theoretic generalization of Galois cohomology. Let $M$ be an atomic and strongly $\omega$-homogeneous structure over a set of parameters $A$. Let $B$ be a normal extension of $A$ in $M$. We show that a short exact sequence of automorphism groups $1 \to \aut(M/B) \to \aut(M/A) \to \aut(B/A) \to 1$  induces a short exact sequence in definable Galois cohomology. We also discuss compatibilities with \cite{pillay2024automorphism}. Our result complements the long exact sequence in definable Galois cohomology developed in \cite{OmarAnandDavid}.
\end{abstract}

\section{Introduction and preliminaries}

%Galois cohomology is a fundamental tool in algebraic number theory. As differetial algebra and differential Galois theory has developed, Ellis Kolchin's constrained (differential) Galois cohomology has solidified its position as a differential cognate to the first nonabelian Galois cohomology set.  
In \textit{Remarks on Galois cohomology and definability} \cite{pillay_1997}, Anand Pillay introduced definable Galois cohomology, a model-theoretic generalization of the first nonabelian Galois cohomology set adapted for suitable first order theories $T$. When $T$ is specialized to the theory $DCF_0$, definable Galois cohomology coincides with constrained differential Galois cohomology as defined in Chapter 7 of Kolchin's \textit{Differential Algebraic Groups} \cite{kolchin1985differential}. When specialized to the theory $ACF_0$, definable Galois cohomology coincides with usual algebraic non-abelian Galois cohomology as defined in Chapter 1 Section 5 of Serre's \textit{Galois Cohomology} \cite{serre1979galois}. We begin by recalling the setting and main result of \cite{pillay_1997} before stating our results. We will continue to work in this setting throughout this paper.

Let $T$ be a complete theory eliminating imaginaries (or consider $T^{eq}$). Let $M$ be a model of $T$ which is atomic and strongly $\omega$-homogeneous over a set of parameters $A$. That is, the type of any finite tuple from $M$ over $A$ is isolated and for any two finite tuples from $M$ with the same type over $A$, there is an automorphism of $M$ taking one tuple to the other. 

Recall from \cite{pillay_1997} that an ($A$-definable) right principal homogeneous space (PHS) for an $A$-definable group $G$ is a nonempty $A$-definable set $X$ together with a right $A$-definable action $\rho:X \times G \to X$ such that $\forall x,y \in X$ there exists a unique $z \in G$ such that $\rho(x,z) = y$. A pointed PHS for $G$ is a pair $(X,p)$ consisting of a PHS for $G$ with a specified $M$-point $p$ of $X$. 

An $A$-definable bijection between two PHSs for $G$ which commutes with the actions of $G$ is an isomorphism of PHSs. The set of isomorphism classes of PHSs for $G$ has the structure of a pointed set with basepoint being the isomorphism class containing $G$ itself considered as a PHS. Altering slightly the notation from \cite{pillay_1997}, we will denote this pointed set $\PP_{\defin}(M/A,G(M))$. 

Similarly, an $A$-definable bijection of two pointed PHSs for $G$ commuting with the actions of $G$ and which sends the basepoint of one to the basepoint of the other is an isomorphism of pointed PHSs. The set of $A$-isomorphism classes of pointed PHSs for $G$ also has the structure of a pointed set. The basepoint is the isomorphism class containing the pointed PHS $(G,e)$ where $e$ is the identity of $G$. We will denote this pointed set $\PP_{\defin,*}(M/A,G(M))$.

A (right) definable cocycle is a function $\varphi$ from $\aut(M/A)$ to $G(M)$, the $M$-points of $G$, such that  firstly, for any $\sigma, \tau \in \aut(M/A)$, $$\varphi(\sigma\tau) = \varphi(\sigma)\sigma(\varphi(\tau))$$ and secondly, there exists an $A$-definable function $h(\bar{x},\bar{y})$ and tuple $\bar{a}$ from $M$ such that $\forall \sigma \in \aut(M/A)$, $$\varphi(\sigma) = h(\bar{a},\sigma(\bar{a})).$$ The set of cocycles from $\aut(M/A)$ to $G(M)$ is denoted $\ZZ^1_{\defin}(M/A,G(M))$ and has the structure of a pointed set. The cocycle which takes constant value at the identity $e \in G(M)$ is the basepoint. 

Two definable cocycles $\varphi$ and $\psi$ are said to be cohomologous if there is an element $b \in G(M)$ such that  $$\varphi(\sigma) = b^{-1}\psi(\sigma)\sigma(b)$$ for all $\sigma \in \aut(M/A)$. Cohomology is an equivalence relation on the pointed set of cocycles, $\ZZ^1_{\defin}(M/A,G(M))$, and a cocycle is called trivial if it is cohomologous to the constant cocycle. Taking the quotient of $\ZZ^1_{\defin}(M/A,G(M))$ by the cohomology equivalence relation yields another pointed set, the pointed set of cohomology classes of cocycles, which is denoted $\HH^1_{\defin}(B/A,G(B))$ and which has the class of trivial cocycles as its basepoint.\\ 

The main result of \cite{pillay_1997} is the following: 

\begin{proposition}(Proposition 3.3. of \cite{pillay_1997})\label{PHS_cocycle_correspondence_aut(M/A)}
    There is a natural isomorphism (basepoint preserving bijection, natural with respect to morphisms of definable groups in the second coordinate), $$\PP_{\defin}(M/A,G(M)) \cong \HH^1_{\defin}(M/A,G(M)),$$ between the pointed set of isomorphism classes of $A$-definable right principal homogeneous spaces for $G$ and $\HH^1_{\defin}(M/A,G(M))$, the pointed set of equivalence classes of definable cocycles from $\aut(M/A)$ to $G(M)$ modulo the cohomology relation.
\end{proposition} 

An extension of parameters $B$, with $A \subseteq B \subseteq M$, is called normal if it is closed under type (in finitely many variables) in $M$ over $A$: Let $\bar{b}$ be a finite tuple from $B$. If $\bar{b}'$ is a finite tuple from $M$ with $\tp(\bar{b}/A) = \tp(\bar{b}'/A)$ then $\bar{b}'$ must also be contained in $B$. By $\aut(B/A)$ we mean the group of elementary permutations of $B$ fixing $A$. In Section 2 of this paper we will review normal extensions and describe some additional properties which they satisfy in the three settings of atomic and strongly $\omega$-homogeneous structures discussed in \cite{pillay_1997}. 

In Section 3 of this paper we will define the definable Galois cohomology set $\HH^1_{\defin}(B/A,G(B))$ where $B$ is a normal extension of $A$ in $M$ and generalize the main result of \cite{pillay_1997} to show that $\HH^1_{\defin}(B/A,G(B))$ classifies exactly the isomorphism classes of PHS for $G$ which contain a $B$-point. 

In \cite{pillay2024automorphism}, Pillay adapted the notion of a (relatively) definable subset of $\aut(M/A)$ from the setting of a saturated model to that of an atomic and strongly $\omega$-homogeneous structure. This yields an equivalent definition of a definable cocycle (see Lemma 3.5 of \cite{pillay2024automorphism}). Also in \cite{pillay2024automorphism}, Pillay introduces (relatively) definable subsets of $\aut(B/A)$ for normal extensions $B$ of $A$ in $M$. At the end of Section 3 below, we will show that our notion of definable cocycle, introduced at the beginning of Section 3 for the setting of a normal extension, has an equivalent phrasing using the definition of definable subset of $\aut(B/A)$ from \cite{pillay2024automorphism}. 

In Section 4 we will prove the main result of this paper: For $B$, a normal extension of $A$ contained in $M$, a short exact sequence $$1 \to \aut(M/B) \to \aut(M/A) \to \aut(B/A) \to 1$$ induces the following short exact sequence in definable Galois cohomology:

\vspace{-4mm}

\[1 \to \HH^1_{\defin}(B/A,G(B)) \to \HH^1_{\defin}(M/A,G(M)) \to \HH^1_{\defin}(M/B,G(M))^{\aut(B/A)}.\] 

We begin section 4 with an adaptation of Serre's ``transform" action of $\aut(B/A)$ on $\HH^1_{\defin}(M/B,G(M))$ to conclude that the image of the final map of the sequence lands in the fixed points of this action.

The short exact sequence of the present paper is a generalization of a classical algebraic statement appearing in Chapter 1 Section 5.8 of Serre's \textit{Galois Cohomology} \cite{serre1979galois}. The result was proved by Kolchin in the differential setting as Theorem 1 of Chapter 7 Section 2 of \cite{kolchin1985differential}. Kolchin's statement does not include mention of the action of $\aut(B/A)$ on the final term so we have strengthened the differential statement slightly. The result is novel in the general definable setting of \cite{pillay_1997}. Together with the long exact sequence of \cite{OmarAnandDavid}, our result brings definable Galois cohomology into a more fully developed state.  

%A future goal of the author is to adapt definable Galois cohomology to apply to the theory $ACFA$ and interface with the model theoretic treatment of difference Galois theory.

\subsection*{Acknowledgements}

The author would like to thank Anand Pillay for many helpful comments and suggestions, and especially for several discussions which led to a simplification of the material in section 2. The author would like to thank the referee for suggesting the inclusion of an adaptation of Serre's transform action on the final term of the short exact sequence.  

\section{Normal extensions}

We continue in the setting of \cite{pillay_1997}. Namely, $T$ is a complete theory eliminating imaginaries. We work in a model $M$ of $T$ which is atomic and strongly $\omega$-homogeneous over a set of parameters $A$.

The following definition precedes Lemma 9.2.6 of \cite{tent_ziegler_2012}.

\begin{definition} 
An extension $A \subseteq B$ of parameters in $M$ is said to be normal if, for any finite tuple $\bar{b}$ from $B$, every realization of $\tp(\bar{b}/A)$ in $M$ is contained in $B$.
\end{definition}

If $B$ is a normal extension of $A$, then $B$ is $\aut(M/A)$-invariant and the restriction map from $\aut(M/A)$ to $\aut(B/A)$, the group of elementary permutations of $B$ fixing $A$, is a group homomorphism with kernel $\aut(M/B)$. So there is a short exact sequence 
$$1 \to \aut(M/B) \to \aut(M/A) \to \aut(B/A).$$

In \cite{pillay_1997}, three model-theoretic settings are discussed in which the conditions on $A \subseteq M$ of atomicity and strong $\omega$-homogeneity are satisfied. These conditions ensure that Proposition \ref{PHS_cocycle_correspondence_aut(M/A)} holds, i.e., that definable Galois cohomology classifies principal homogeneous spaces for definable groups. 

In the first setting, $M$ is taken to be the prime model over $A$ of an $\omega$-stable theory. In the second setting, $M$ is taken to be $\acl(A)$. In the third setting, $M$ is the countable model of an $\omega$-categorical theory. We will now address the assumptions on $B$ needed in each setting so that the final map in the above sequence is surjective, that is, every elementary permutation of $B$ fixing $A$ extends to an automorphism of $M$ fixing $A$.

When $M$ is the model-theoretic algebraic closure of $A$ or the prime model over $A$ of an $\omega$-stable theory, then for any normal extension $B$ in $M$ over $A$ the restriction map from $\aut(M/A)$ to $\aut(B/A)$ is surjective. In the $\omega$-categorical setting $B\backslash A$ must be finite for the restriction map to be surjective.

\begin{remark}
    We slightly generalize the first setting of an $\omega$-stable theory $T$ to a totally transcendental theory $T$. Note that any $\omega$-stable theory is totally transcendental, the converse being true when $T$ is countable.
\end{remark}

To see the surjectivity of the restriction map in the setting of the prime model of a totally transcendental theory, we use a fundamental characterization of such models which can be found in chapter 9 of \cite{tent_ziegler_2012}: 

\begin{fact}\label{primemodelchar}
    Let $T$ be a totally transcendental theory. A model $M$ of $T$ is prime over a set of parameters $A$ if and only if it is atomic over $A$ and contains no uncountable $A$-indiscernible sequence.
\end{fact}

\begin{lemma}
    Let $M$ be the prime model over $A$ of a totally transcendental theory $T$. Let $B$ be a normal extension over $A$ in $M$. Then
    \begin{enumerate}[(i)]
        \item $M$ is atomic over $B$.
        \item $M$ is prime over $B$.
        \item The restriction homomorphism $\aut(M/A) \to \aut(B/A)$ is surjective.
    \end{enumerate}
\end{lemma}

\begin{proof}
    \begin{enumerate}[(i)]
        \item This can also be found in chapter 9 of \cite{tent_ziegler_2012}.  Note that the isolated types are dense in the type space over arbitrary parameter sets in totally transcendental theories. 
        
        Let $\bb \in M^n$. Let $\phi(\xx)$ isolate $\tp(\bb/A)$. Since the isolated types are dense over $B$, there is a formula $\psi(\xx,\dd)$ isolating a complete type over $B$ which contains the formula $\phi(\xx)$.  Let $\bb'$ realize $\psi(\xx,\dd)$ in $M$. Then $\bb'$ realizes $\phi(\xx)$ and so $\bb$ and $\bb'$ have the same type over $A$. By the strong $\omega$-homogeneity of $M$ there is an automorphism $\sigma$ of $M$ taking $\bb'$ to $\bb$. Let $\ee = \sigma(\dd)$, which by normality must lie in $B$. Then $\psi(\xx,\ee)$ isolates the type of $\bb$ over $B$.
        \item As $A \subseteq B$, any $B$-indicernible sequence is $A$-indicernible. As $M$ is prime over $A$, it contains no uncountable $A$-indicernible sequence and therefore no uncountable $B$-indicernible sequence. This, together with (i), shows from Fact \ref{primemodelchar} that $M$ is prime over $B$. 
        \item Let $\sigma$ be an elementary permutation of $B$ over $A$. Let $\bar{M}$ be a sufficiently saturated and strongly homogeneous model of $T$. Then we can assume that $\bar{M}$ contains $M$ and that $\sigma$ extends to an automorphism $\gamma\in \aut(\bar{M}/A)$. Let $M' = \gamma(M)$. By the characterization of prime models, $M'$ is still the prime model over $B$. By the uniqueness of prime models of an $\omega$-stable theory, there is an isomorphism $\tau:M' \to M$ fixing $B$. Then $\tau \circ \gamma|_M:M \to M$ is an automorphism of $M$ extending $\sigma$.
    \end{enumerate}
\end{proof}

%\begin{remark}
%    There is another way to see the surjectivity of the restriction homomorphism in the setting of a prime model. Recall that in an $\omega$-stable or totally transcendental theory, a prime model over a set of parameters is the unique constructible model over that set. 
    
%    The argument of Part 1 of the previous lemma can be used to show directly that if $M$ is constructible over $A$, and $B$ is a normal intermediate set of parameters, then $M$ remains constructible and hence prime over $B$. Let $\sigma \in \aut(B/A)$ extend to an automorphism $\gamma \in \aut(\bar{M}/A)$. Then $\gamma(M)$ is still constructible over $A$, and by the adaptation of the argument from Part 1, still constructible over $B$. Ressayre's theorem proving the uniqueness of constructible models yields an isomorphism $\tau:\gamma(M) \to M$ fixing $B$. As before $\tau \circ \gamma|_M:M \to M$ is an automorphism of $M$ extending $\sigma$.
%\end{remark}

\begin{proposition}
    Let $B$ be normal in $M$ over $A$ where $M$ is $\acl(A)$ or the prime model of a totally transcendental theory or where $M$ is $\omega$-categorical and additionally we require $B\backslash A$ to be finite. Then there is a short exact sequence of groups: $$1 \to \aut(M/B) \to \aut(M/A) \to \aut(B/A) \to 1$$
\end{proposition}

\begin{proof}
    In the $\omega$-categorical case the surjectivity follows from the strong $\omega$-homogeneity of $M$ and the fact that we have taken $B\backslash A$ to be finite. In the case where $M = \acl(A)$, the surjectivity follows from a standard back and forth argument. The preceding lemma covers the case of the prime model of a totally transcendental theory. 
\end{proof}

\section{Definable Galois cohomology for normal extensions}

We continue to work in the setting of \cite{pillay_1997}. Let $T$ be a complete theory eliminating imaginaries (EI). Let $M$ be a model of $T$ containing a set of parameters $A$, over which $M$ is atomic and strongly $\omega$-homogeneous. 

Let $B$ be a normal extension of $A$ in $M$ such that elementary permutations of $B$ over $A$ (which we denote $\aut(B/A)$) extend to automorphisms of $M$ over $A$ and such that $M$ is atomic and strongly $\omega$-homogeneous over $B$. We have in mind the three model-theoretic settings of \cite{pillay_1997} discussed in the last section: i.e. when $M$ is the prime model over $A$ of a totally transcendental theory or $M$ is $\acl(A)$ then any normal intermediate extension $B$ will satisfy these conditions. When $M$ is the countable model of an $\omega$-categorical theory, $B\backslash A$ must be finite in addition to $B$ being normal.  

By the atomicity and strong $\omega$-homogeneity of $M$ over $B$, we have that an $M$-definable $\aut(M/B)$-invariant set is $B$-definable, from which it follows that the fixed set of the action of $\aut(M/B)$ on $M$ is $\dcl(B)$.  

\begin{definition}
    Let $G$ be an $A$-definable group and $X$ an $A$-definable set. Let $\varphi(\bar{x})$ and $\psi(\bar{y})$ be $A$-formulas defining the underlying set of $G$ and $X$ respectively. By $G(B)$ we mean $$\{g \in B^{|\bar{x}|} :  M \models \varphi(g) \}$$ or equivalently, $G(M) \cap B^{|\bar{x}|}$.  Similarly, by $X(B)$ we mean $$\{p \in B^{|\bar{y}|} : 
 M  \models \psi(p)\}$$ equivalently, $X(M)\cap B^{|\bar{y}|}$. We will refer to elements of $G(B)$ and $X(B)$ as $B$-points of $G(M)$ and $X(M)$.
\end{definition}

We will also now require $B = \dcl(B)$ for two reasons, firstly, so that $B$ is exactly the fixed set of the action of $\aut(M/B)$ on $M$ and secondly, so that $G(B)$ is a subgroup of $G(M)$.

\begin{lemma}
    Let $G$ be an $A$-definable group. Let $A \subseteq B \subseteq M$ with $B = \dcl(B)$. Then $G(B)$ is an abstract subgroup of $G(M)$. Moreover, if $X$ is an $A$-definable principal homogeneous space for $G$ which contains a $B$-point, then $X(B)$ is an abstract principal homogeneous space for $G(B)$ equivariant with respect to the left action of $\aut(B/A)$. 
\end{lemma}

\begin{proof}
    Again, let $\varphi(\bar{x})$ and $\psi(\bar{y})$ be $A$-formulas defining the underlying set of $G$ and $X$ respectively.
    
    Let $e \in G$ be the identity. As $G$ is $A$-definable, $e \in G(\dcl(A)) \subseteq G(B)$. Let $g_1,g_2 \in G(B)$. Their product $g_1g_2$ is definable over $A\cup \{g_1, g_2\}$, that is $g_1g_2 \in \dcl(A,g_1,g_2)$ and so must lie in $\dcl(B)$ as $\dcl(B) = B$. We have $g_1g_2 \in B^{|\bar{x}|} \cap G(M) = G(B)$. Likewise, $g_1^{-1}$ is definable over $\dcl(A,g_1)$, so $g_1^{-1} \in B^{|\bar{x}|} \cap G(M) = G(B)$. Hence $G(B)$ is an abstract subgroup of $G(M)$.
    
    By assumption $X(B) \neq \emptyset$, so let $p \in X(B)$. 
    Let $g \in G(B)$, then $pg \in \dcl(A,p,g)$. As $\dcl(A,p,g) \subseteq \dcl(B) = B$, $pg \in B^{|\bar{y}|}$. Hence $pg \in X(B)$. So $G(B)$ acts on $X(B)$. Now let $p' \in X(B)$. As $X$ is an $A$-definable PHS for $G$, there is a unique element $g \in G(M)$ such that $pg = g'$. So $g \in \dcl(A,p,p')$. As $\dcl(A,p,p') \subseteq \dcl(B) = B$, $g \in B^{|\bar{x}|}$. Hence $g \in G(B)$. So $X(B)$ is an abstract PHS for $G(B)$. 

    Let $\sigma \in \aut(B/A)$. Let $p \in X(B)$ and $g \in G(B)$ then $\sigma(pg) = \sigma(p)\sigma(g)$ as the action is $A$-definable and $\sigma$ is an $A$-elementary permutation of $B$. So the equivariance of the action is verified.
\end{proof}

In this section we will define definable Galois cohomology in the setting of a normal extension $A \subseteq B$ and introduce the notation $\HH^1_{\defin}(B/A,G(B))$. We then introduce some further notation for certain pointed sets of isomorphism classes of definable PHSs. We then prove that $\HH^1_{\defin}(B/A,G(B))$ classifies exactly the isomorphism classes of definable PHSs for $G$ which contain a $B$-point. This refines Proposition \ref{PHS_cocycle_correspondence_aut(M/A)}, the main result of \cite{pillay_1997}, to apply in the setting of normal intermediate extensions. 

\begin{definition}\label{def coycle}
    A definable (right) cocycle $\varphi:\aut(B/A) \to G(B)$ is a function $\varphi$ together with an $A$-definable function $h(\xx,\yy)$ and a tuple $\aaa$ from $B$ satisfying $\forall \sigma, \tau \in \aut(B/A)$,

\begin{enumerate}[(i)]
    \item $\varphi(\sigma\tau) = \varphi(\sigma)\sigma(\varphi(\tau))$ (cocycle condition)
    \item $\varphi(\sigma) = h(\aaa,\sigma(\aaa))$ (definability condition)
\end{enumerate}

The definable cocycles from $\aut(B/A)$ to $G(B)$ form a pointed set, which we denote $Z^1_{\defin}(B/A,G(B))$. We define the basepoint of $Z^1_{\defin}(B/A,G(B))$ to be the constant cocycle which sends every elementary permutation to the identity of $G$. 

Two definable cocycles $\varphi$ and $\psi$ are said to be cohomologous if there is an element $b \in G(B)$ such that  $$\varphi(\sigma) = b^{-1}\psi(\sigma)\sigma(b)$$ for all $\sigma \in \aut(B/A)$. As we required $B$ to be $\dcl$-closed, $G(B)$ is a group and cohomology is again an equivalence relation on $\ZZ^1_{\defin}(B/A,G(B))$. Again we call a cocycle trivial if it is cohomologous to the constant cocycle. Trivial cocycles are therefore of the form $$\varphi(\sigma) = b^{-1}\sigma(b)$$ for some $b \in G(B)$. The quotient of $\ZZ^1_{\defin}(B/A,G(B))$ via the cohomology equivalence relation is again a pointed set which we denote $\HH^1_{\defin}(B/A,G(B))$ and which has the class of trivial cocycles as its basepoint. 
\end{definition}

We now describe the relationship between the pointed sets $\ZZ^1_{\defin}(B/A,G(B))$ and $\ZZ^1_{\defin}(M/A,G(M))$. Let $\pi:\aut(M/A) \to \aut(B/A)$ be the projection homomorphism which via our assumptions is surjective. 

\begin{lemma}\label{ZZ inj}
    Let $\varphi \in \ZZ^1_{\defin}(B/A,G(B))$, then $$- \circ\pi: \ZZ^1_{\defin}(B/A,G(B)) \to \ZZ^1_{\defin}(M/A,G(M))$$ defined by $\varphi \mapsto \varphi \circ \pi$ is an injective map of pointed sets whose image is exactly the set of definable cocycles in $\ZZ^1_{\defin}(M/A,G(M))$ which have definability parameter $\bar{a}$ in $B$. Moreover $-\circ \pi$ descends to an injective map on cohomology, $$\pi^1: \HH^1_{\defin}(B/A,G(B)) \to \HH^1_{\defin}(M/A,G(M)).$$
\end{lemma}

\begin{proof}
Let $\varphi$ be a definable cocycle from $\aut(B/A)$ to $G(B)$ with definability data $h(\bar{x},\bar{y})$ and $\bar{a}$ in $B$. We first check that $\varphi \circ \pi$ is still a cocycle witnessed by the same definability data:\\

Let $\sigma, \tau \in \aut(M/A)$, 
\begin{align*}
    \varphi \circ \pi(\sigma\tau) &= \varphi(\pi(\sigma)\cdot \pi(\tau))\\
    &= \varphi(\pi(\sigma))\cdot \pi(\sigma)(\varphi(\pi(\tau)))\\
    &= \varphi\circ \pi(\sigma)\cdot \sigma(\varphi\circ \pi(\tau))
\end{align*}

So the cocycle condition is satisfied. For the definability condition, let $\sigma \in \aut(M/A)$. As $\bar{a}$ is in $B$, $$\varphi\circ \pi(\sigma) = h(\bar{a},\pi(\sigma)(\bar{a})) = h(
\bar{a},\sigma(\bar{a})).$$

Since $\pi$ is surjective, the map $-\circ \pi$ on cocycles is injective.\\

Let $\varphi \in \ZZ^1_{\defin}(M/A,G(M))$ with definability parameter $\bar{a}$ coming from $B$. Let $\sigma \in \aut(M/A)$. The definability condition $\varphi(\sigma) = h(\bar{a},\sigma(\bar{a}))$ shows that $\varphi(\sigma) \in \dcl(\bar{a}, \sigma(\bar{a}))$. By the normality of $B$ and the assumption that $B$ is $\dcl$-closed, we obtain $\varphi(\sigma) \in G(B)$. Furthermore, as $\bar{a}$ comes from $B$ and $\varphi(\sigma) = h(\bar{a},\sigma(\bar{a}))$, $\varphi(\sigma)$ depends only on the restriction $\pi(\sigma)$ and so descends to a well defined map on $\aut(B/A)$ which we call $\tilde{\varphi}$. As $\tilde{\varphi}$ and $\varphi$ have the same image, it is clear that $\tilde{\varphi} \in \ZZ^1_{\defin}(B/A,G(B))$ as witnessed by the same definability data. Lastly, $\tilde{\varphi}\circ \pi = \varphi$. So the image of $\ZZ^1_{\defin}(B/A,G(B))$ under $-\circ\pi$ is exactly the subset of definable cocycles in $\ZZ^1_{\defin}(M/A,G(M))$ which have definability parameter $\bar{a}$ in $B$.\\

Let $\varphi_1$ and $\varphi_2$ in  $Z^1_{\defin}(B/A,G(B))$ be cohomologous as witnessed by $g \in G(B)$. As $G(B) \subseteq G(M)$, the same element $g$ witnesses that $\varphi_1 \circ \pi$ and $\varphi_2 \circ \pi$ are cohomologous. So the map on cocycles descends to a well defined map on cohomology which we call $\pi^*$.\\

Let $\varphi_1$ and $\varphi_2$ be cocycles from $\aut(B/A)$ to $G(B)$ with definability data $h_1$, $\bar{a}_1$ and $h_2$, $\bar{a}_2$ respectively. Suppose that $\varphi_1 \circ \pi$ and $\varphi_2\circ\pi$ are cohomologous via $g \in G(M)$. So for all $\sigma \in \aut(M/A)$, $$g^{-1}h_1(\bar{a}_1,\sigma(\bar{a}_1))\sigma(g) = h_2(\bar{a}_2,\sigma(\bar{a}_2)).$$ As $\bar{a}_1$ and $\bar{a}_2$ lie in $B$, for all $\sigma \in \aut(M/B)$ we have $h(\bar{a}_1,\sigma(\bar{a}_1)) = e$ and $h(\bar{a}_2,\sigma(\bar{a}_2)) = e$  and therefore $g^{-1}\sigma(g) = e$. So $g$ is fixed by the action of $\aut(M/B)$. Then as $B = \dcl(B)$, we have $g \in G(B)$. So the same $g$ witnesses that $\varphi_1$ and $\varphi_2$ are cohomologous. We have then shown that the induced map on cohomology, $\pi^*$, is injective.
\end{proof}

We now alter and extend slightly the notation from \cite{pillay_1997} for pointed sets of isomorphism classes of definable PHSs. We first make some observations. 

As $B$ is $\dcl$-closed, if $X_1$ and $X_2$ are isomorphic $A$-definable PHSs for an $A$-definable group $G$ then $X_1$ has a $B$-point if and only if $X_2$ has a $B$-point. Moreover, if $(X_1,p_1)$ and $(X_2,p_2)$ are isomorphic pointed PHS for $G$, $p_1$ is a $B$-point of $X_1$ if and only if $p_2$ is a $B$-point of $X_2$.

\begin{definition}
    The $A$-isomorphism classes of $A$-definable pointed PHS for $G$ whose representatives all have a $B$-point as their basepoint form a sub-pointed set of $\PP_{\defin}(M/A,G(M))$ which we will denote $\PP_{\defin,*}(B/A,G(B))$. 
    
    In addition, the set of $A$-isomorphism classes of $A$-definable PHS for $G$ whose representatives all contain a $B$-point form a sub-pointed set of $\PP_{\defin}(M/A,G(M))$ which we denote $\PP_{\defin}(B/A,G(B))$. 
\end{definition}

Note that the identity of $G$ is an $A$-point, so contained in $B$. So the isomorphism class containing $G$ is in $\PP_{\defin}(B/A,G(B))$. Note also that the isomorphism class of $(G,e)$ is in $\PP_{\defin,*}(B/A,G(B))$. Thus $\PP_{\defin,*}(B/A,G(B))$ and $\PP_{\defin}(B/A,G(B))$ are nonempty.

\begin{remark}\label{Bpts of PHS}
    As $B$ is required to be $\dcl$-closed, if $X$ is any $A$-definable right PHS for $G$ with a $B$-point, then restricting the action of $G$ on $X$ to $B$-points, we see that $X(B)$ is an abstract right PHS for $G(B)$ equivariant with respect to the left action of $\aut(B/A)$.
\end{remark}

We now refine Proposition \ref{PHS_cocycle_correspondence_aut(M/A)} to describe the setting of normal extensions of parameters.

\begin{proposition}\label{CocyclePHScorrespB}
    There is a natural isomorphism $$\PP_{\defin}(B/A,G(B)) \cong \HH^1_{\defin}(B/A,G(B))$$ between the set of cohomology classes of definable cocycles from $\aut(B/A)$ to $G(B)$ and the pointed set of isomorphism classes of $A$-definable PHS $X$ for $G$ which contain a $B$-point.
\end{proposition}

The following lemmas trace the constructions (from \cite{pillay_1997}) of a cocycle in $\ZZ^{1}_{\defin}(M/A,G(M))$ from a pointed PHS in $\PP_{\defin,*}(M/A,G(M))$ and vice versa. We check that via these constructions the cocycles in the image of $\ZZ^{1}_{\defin}(B/A,G(B))$ under the map of Lemma \ref{ZZ inj} corresponds to $\PP_{\defin,*}(B/A,G(B))$, that is, cocycles whose definability parameters come from $B$ correspond exactly to pointed PHSs which have a $B$-point as their basepoint.

\begin{lemma}\label{pPHS to cocycle}
    An $A$-definable pointed PHS $(X,p)$ for $G$ with $p \in X(B)$ gives rise to a cocycle $\varphi:\aut(M/A) \to G(M)$ which has the $B$-point $p \in X(B)$ as its definability parameter.
\end{lemma}

\begin{proof}
    As $p \in X(B)$, $\sigma(p) \in X(B)$. We define $\varphi$ by sending $\sigma \in \aut(B/A)$ to the unique $g_\sigma \in G(B)$ such that $\sigma(p) = pg_\sigma$. The cocycle condition is easily checked. Let $\Gamma(\mu)(\bar{x},\bar{z},\bar{y})$ define the graph of the definable action $\mu:X \times G \to X$. Let $h(\bar{x},\bar{y}) = \bar{z}$ be the definable function whose graph is $\Gamma(\mu)(\bar{x},\bar{z},\bar{y})$. Then $h(\bar{x},\bar{y})$ together with the parameter $p$ witnesses the definability condition.
\end{proof}

Let $\varphi:\aut(M/A) \to G(M)$ be a definable cocycle witnessed by an $A$-definable function $h$ and definability parameter $\bar{a}$ from $B$. Note that as $B \subseteq M$ and $M$ is atomic over $A$, the type $\tp(\bar{a}/A)$ is isolated and is then given by an $A$-definable set $Z$ defined by a formula $\theta(\bar{x})$. The following facts are from \cite{pillay_1997}.   

\begin{fact}\label{cocycle consequence for h}
     Let $\bar{b}$ and $\bar{c}$ be tuples from $M$ having the same type as $\bar{a}$ over $A$. Then $h(\bar{a},\bar{c}) = h(\bar{a},\bar{b})h(\bar{b},\bar{c})$. 
\end{fact}

\begin{fact}
    Let $(\bar{a}_1,g_1)$ and $(\bar{a}_2,g_2)$ be elements from $Y(M):= Z(M) \times G(M)$. The binary relation defined by $h(\bar{a}_1,\bar{a}_2)g_2 = g_1$ is an $A$-definable equivalence relation on this set. We will denote the equivalence relation by $E$.
\end{fact}

\begin{lemma}\label{cocycle to pPHS}
    The definable cocycle $\varphi$, with definability parameter contained in $B$, gives rise to a pointed PHS for $G$ whose basepoint is a $B$-point.
\end{lemma}

\begin{proof}
    By EI, the quotient of $Y = Z \times G$ by $E$, is again an $A$-definable set $X$. 

    It is clear that $G$ acts $A$-definably on the right on $Y$ by multiplication on the right on the $G$-coordinate. It is easy to see that this action preserves the classes of the equivalence relation and so descends to an action on $X$: if $g \in G$, $h(\bar{a}_1,\bar{a}_2)g_2 = g_1$ iff $h(\bar{a}_1,\bar{a}_2)g_2g = g_1g$. Similarly, because $G$ is a right torsor for itself, the action is regular: Let $(\bar{a}_1,g_1)$ and $(\bar{a}_2,g_2)$ be elements of $Y$. As $h(\bar{a}_1,\bar{a}_2)$, $g_1$ and $g_2$ are elements of $G$, there is a unique element $g \in G$ such that $h(\bar{a}_1,\bar{a}_2)g_2g = g_1$. That is, $g$ is the unique element such that $(\bar{a}_1,g_1)E(\bar{a}_2,g_2g)$. We have shown that $X$ is an $A$-definable PHS for $G$. 
    
    Finally, the specified basepoint of $X$ is the $E$-class of $(\bar{a},e)$ which lives in $\dcl^{eq}(B)$. Note by EI and assumptions, $\dcl^{eq}(B) = \dcl(B) = B$.
\end{proof}

Before proving Proposition \ref{CocyclePHScorrespB} we discuss the correspondence of Proposition \ref{PHS_cocycle_correspondence_aut(M/A)}. Cocycles $\varphi:\aut(M/A) \to G(M)$ correspond to isomorphism classes of pointed PHSs for $G$. Cohomology classes of cocycles correspond to isomorphism classes of PHSs for $G$, that is, cohomologous cocycles correspond to pointed PHSs which are isomorphic as PHSs but not as pointed PHSs: There is an $A$-definable $G$-equivariant bijection between the PHSs which is not neccesarily basepoint preserving.

\begin{proof}[Proof of Proposition \ref{CocyclePHScorrespB}]
    By Lemma 3.2, $\HH^1(B/A,G(B))$ injects into  the cohomology set $\HH^1(M/A,G(M))$. The image of this map is exactly the set of cohomology classes containing a cocycle with definability parameter contained in $B$. 
    
    By Lemmas 3.6 and 3.9, cocycles with definability parameters contained in $B$ correspond exactly with pointed PHSs for $G$ whose basepoint is a $B$-point. From the preceding discussion, cocycles which are cohomologous to cocycles with definability parameters from $B$ correspond to pointed PHSs which are isomorphic as PHSs (but not neccesarily as pointed PHS) to PHSs which have a $B$-point as basepoint, and so themselves must contain a $B$-point. 
    
    We have shown that $\PP_{\defin}(B/A,G(B))$, the pointed set of isomorphism classes of PHSs whose representatives contain a $B$-point, is isomorphic as a pointed set to $\HH^{1}_{\defin}(B/A,G(B))$. 
\end{proof}

We conclude this section by giving an alternative definition of a definable cocycle from $\aut(B/A)$ to $G(B)$ using notions introduced in \cite{pillay2024automorphism} and remark on the various definitions of cocycles appearing in the literature. 

Recall from \cite{pillay2024automorphism} that a set $Y \subseteq \aut(M/A)$ is called definable if there is an $A$-formula $\varphi(\bar{x},\bar{y})$ and a tuple $\bar{b}$ from $M$ such that $$Y = \{\sigma \in \aut(M/A) :  M \models \varphi(\bar{b},\sigma(\bar{b}))\}.$$ Furthermore, let $X$ be an $A$-definable set. We say that $\Lambda$, a subset of $\aut(M/A)\times X(M)$, is definable if there is an $A$-formula $\varphi(\bar{x},\bar{y},\bar{z})$ and a tuple $\bar{b}$ from $M$ such that $$\Lambda = \{(\sigma,\bar{c}) \in \aut(M/A)\times X(M
): M \models \varphi(\bar{b},\sigma(\bar{b}),\bar{c}))\}.$$

The following statement is Lemma 3.5 of \cite{pillay2024automorphism}. It gives either an alternative definition or a characterization of definable cocycles: Let $\varphi$ be a function from $\aut(M/A)$ to $G$. Then there is an $A$-definable function $h(\bar{x},\bar{y})$ and a tuple $\bar{b}$ from $M$ such that for all $\sigma \in \aut(M/A)$ $\varphi(\sigma) = h(\bar{b},\sigma(\bar{b}))$ if and only if the graph of $\varphi$ is definable in the above sense. 

We now give a similar statement in the setting of a normal extension.  In \cite{pillay2024automorphism}, the following definitions (or a slight variation of them depending on the version) are given: 

\begin{definition}
    A set $Y \subseteq \aut(B/A)$ is called definable if there is an $A$-formula $\varphi(\bar{x},\bar{y})$ and a tuple $\bar{b}$ from $B$ such that $$Y = \{\sigma \in \aut(B/A) :  M \models \varphi(\bar{b},\sigma(\bar{b}))\}.$$ Furthermore, let $X$ be an $A$-definable set. We say that $\Lambda$, a subset of $\aut(B/A)\times X(M)$, is definable if there is an $A$-formula $\varphi(\bar{x},\bar{y},\bar{z})$ and a tuple $\bar{b}$ from $B$ such that $$\Lambda = \{(\sigma,\bar{c}) \in \aut(B/A)\times X(M) : M \models \varphi(\bar{b},\sigma(\bar{b}),\bar{c}))\}.$$
\end{definition}

This gives an alternative characterization of the cocycles introduced in the beginning of this section:

\begin{lemma}
    Let $\varphi$ be a function from $\aut(B/A)$ to $G(B)$.  Then there is an $A$-definable function $h(\bar{x},\bar{y})$ and a tuple $\bar{b}$ from $B$ such that for all $\sigma \in \aut(B/A)$ $\varphi(\sigma) = h(\bar{b},\sigma(\bar{b}))$ if and only if the graph of $\varphi$ is definable in the sense of the previous definition.
\end{lemma}

We adapt the proof of Lemma 3.5 of \cite{pillay2024automorphism}.

\begin{proof}
    Suppose that for $\varphi:\aut(B/A) \to G(B)$, there is a tuple $\bar{b}$ from $B$ and definable function $h(\bar{x},\bar{y}) = \bar{z}$ such that for all $\sigma \in \aut(B/A)$, $\varphi(\sigma) = h(\bar{b},\sigma(\bar{b}))$. Let $\Gamma(h)(\bar{x},\bar{y},\bar{z})$ be the graph of $h$. Now, $(\sigma,g) \in \Gamma(\varphi)$, the graph of $\varphi$, iff $\varphi(\sigma)  = g$ iff $ M \models h(\bar{b},\sigma(\bar{b})) = g$. So the graph of $\varphi$ is definable via $\Gamma(h)(\bar{x},\bar{y},\bar{z})$ and parameter $\bar{b}$.

    Suppose now that $\varphi:\aut(B/A) \to G(B)$ is a function whose graph is definable via an $A$-formula $\psi(\bar{x},\bar{y}, \bar{z})$ and parameter $\bar{b}$ from $B$ so that $$\Gamma(\varphi) = \{(\sigma,g) \in \aut(B/A) \times G(M) :  M \models \psi(\bar{b},\sigma(\bar{b}),g)\}.$$ 
    
    Let $\theta(\bar{x})$ isolate the type $\tp(\bar{b}/A)$ and let $Z$ be the associated definable set of realizations of this type. We want to produce a definable function from $h:Z^2 \to G$ satisfying the definability condition of Definition \ref{def coycle} (ii). 
    
    Let $\bar{b}'$ realize the type of $\bar{b}$ over $A$. By the strong $\omega$-homogeneity of $M$ over $A$ and normality of $B$, there is an automorphism $\sigma \in \aut(B/A)$ with $\sigma(\bar{b}) = \bar{b}'$. Let $g = \varphi(\sigma)$. Then by the definability of the graph of $\varphi$,  $M \models \psi(\bar{b},\sigma(\bar{b}),g)$ and so $M \models \psi(\bar{b},\bar{b}',g)$. So for any realization $\bar{b}'$ of $\tp(\bar{b}/A)$ there is at least one $g \in G(B)$ with $M \models \psi(\bar{b},\bar{b'},g)$. Also, let $g,g' \in G(M)$ such that for some $\bar{b}'$ realizing $\tp(\bar{b}/A)$ we have $M \models \psi(\bar{b},\bar{b}',g)$ and $M \models \psi(\bar{b},\bar{b}',g')$. Again, via strong $\omega$-homogeneity and normality let $\sigma(\bar{b}) = \bar{b}'$. Then $M \models \psi(\bar{b},\sigma(\bar{b}),g)$ and $M \models \psi(\bar{b},\sigma(\bar{b}),g')$. By the definability of $\varphi$ we have $(\sigma,g) \in \Gamma(\varphi)$ and $(\sigma,g') \in \Gamma(\varphi)$ and since $\varphi$ is a function, $g = g'$. So there is exactly one $g \in G(M)$ such that $M \models \psi(\bar{b},\bar{b}',\bar{g})$ for any realization of $\tp(\bar{b}/A)$. We therefore have that $$M \models \forall \bar{y} \  \exists^{=1}  \bar{z} \ \theta(\bar{y})\wedge \psi(\bar{b},\bar{y},\bar{z}).$$ Also as $\theta(\bar{x})$ isolates the type of $\bar{b}$, we have $$M \models \forall \bar{x}  \ \forall \bar{y}  \ \exists^{=1} \bar{z}  \ \theta(\bar{y})\wedge\theta(\bar{x}) \wedge \psi(\bar{x},\bar{y},\bar{z}).$$ So $\theta(\bar{y})\wedge\theta(\bar{x}) \wedge \psi(\bar{x},\bar{y},\bar{z})$ is the graph of the desired definable function $h:Z^2 \to G$ with parameter $\bar{b}$ from $B$ satisfying the definability condition of Definition \ref{def coycle} (ii).
\end{proof}

\begin{remark}
    Concerning the various definitions of cocycles:

    \begin{enumerate}
        \item In the algebraic setting of \cite{serre1979galois}, cocycles are defined to be continuous. If $\aut(M/A)$ is given the usual structure of a topological group where a neighborhood base of the identity is the set of stabilizers of finite tuples from $M$, and $G(M)$ is given the discrete topology, then continuity of a cocycle $\varphi$ is equivalent the cocycle condition together with the data of an $\aut(M/A)$-invariant function $h(\bar{x},\bar{y})$ and a tuple $\bar{a}$ from $M$ such that for any $\sigma \in \aut(M/A)$, $\varphi(\sigma) = h(\bar{a},\sigma(\bar{a}))$. This observation is contained implicitly in Lemma 2.6 of \cite{pillay_1997}. This gives some motivation for the definition of a definable cocycle.
        
        In $ACF_0$, for $M = \acl(A)$ the function $h(\bar{x},\bar{y})$ is $\aut(M/A)$-invariant if and only if it is $A$-definable so ``invariant" and definable cocycles coincide. That is, at the level of cocycles, $Z^{1}_{\defin}$ in the setting of $ACF_0$ and $Z^1$ coincide.

        \item Kolchin's constrained cohomology is defined in terms of what he calls constrained cocycles, satisfy a generalization of the continuity condition. He shows that these are closely connected to another set of cocycles called differential rational cocycles whose definition is closer to the definability data definition. The key point, see Remark 2.8 (ii) of \cite{pillay_1997}, is that despite the fact that the set of constrained cocycles defined by Kolchin, $Z^1_{\Delta}$, is a priori slightly different than the set of definable cocycles specialized to the theory $DCF_0$, $Z^1_{\defin}$, the resulting sets of cohomology classes both classify differential algebraic PHSs and so $H^1_{\defin}$ for the theory $DCF_0$ and $H^1_{\Delta}$ are equivalent i.e., isomorphic as pointed sets. 
        
        \item We also remark that Kolchin suppresses the set of rational points on the coefficients in his notation, writing $H^1_\Delta(L/K,G)$ instead of $H^1_\Delta(L/K,G(L))$ for example.
    \end{enumerate}
    
\end{remark}

\section{The short exact sequence in definable Galois cohomology}

In this section we show that a short exact sequence of automorphism groups arising from a normal intermediate extension of parameters induces a short exact sequence in definable Galois cohomology.

We will retain the same assumptions on $M$, $A$, $T$, and $G$ from the previous sections. Our statement here is slightly more general than that of the introduction; we generalize $M$ in the earlier statement to an additional intermediate normal extension $C$.

Let $A \subseteq B \subseteq C \subseteq M$ with $B$ and $C$ normal in $M$, $\dcl^{eq}$-closed, such that every elementary permutation of either $B$ or $C$ fixing $A$ lifts to an automorphism of $M$ and such that $M$ is atomic and strongly $\omega$-homogeneous over both $B$ and $C$.

\begin{definition}
    Let $\sigma \in \aut(C/A)$. Let $\varphi \in \ZZ^{1}_{\defin}(C/B,G(C))$ be a definable cocycle. Let $(P,p)$ be a pointed definable PHS representing a class in $\mathcal{P}_{\defin,*}(C/B,G(B))$. We define the transform of $\varphi$ by $\sigma$, denoted $\sigma(\varphi)$, to be the map $$\sigma(\varphi)(\tau) := \sigma(\varphi(\sigma^{-1}\tau\sigma))$$ for $\tau \in \aut(B/C)$. We define the the transform of $(P,p)$ by $\sigma$, denoted $\sigma(P,p)$ to be the definable pointed PHS $(\sigma(P),\sigma(p))$.
\end{definition}

\begin{lemma}
    The transform induces a well defined action of $\aut(B/A)$ on both $\HH^1_{\defin}(C/B,G(C))$ and $\mathcal{P}_{\defin}(C/B,G(C))$ such that the isomorphism between them is equivariant with respect to this action.
\end{lemma}

\begin{proof}
    Let $(P,p)$ represent an isomorphism class in $\mathcal{P}_{\defin,*}(C/B,G(B))$.  Let $\sigma \in \aut(C/A)$. By the normality of $B$, $\sigma(P)$ is still a $B$-definable PHS for $G$ containing a $B$-point $\sigma(p)$.

    Let $(Q,q)$ be $B$-definably isomorphic as a pointed PHS to $(P,p)$ via $\Phi$. Then again by normality, $\sigma(\Phi)$ is a $B$-definable PHS isomorphism between $(\sigma(Q),\sigma(q))$ and $(\sigma(P),\sigma(p))$ so the transform is well defined on $\mathcal{P}_{\defin,*}(C/B,G(B))$. 

    Let $(P,p)$ correspond to a definable cocycle $\varphi \in \ZZ^{1}_{\defin}(C/B,G(C))$ via Lemma \ref{pPHS to cocycle}. We show that the transform of the cocycle $\varphi$ by $\sigma$ is the cocycle corresponding to the transform of $(P,p)$. Let $\tau \in \aut(C/B)$. Recall that $\sigma(p)$ is the base point of $\sigma(P)$, so applying $\tau$ we obtain \begin{align*}\tau(\sigma(p)) &= \sigma(\sigma^{-1}(\tau(\sigma(p))))\\
    &= \sigma(\sigma^{-1}\tau\sigma(p))\\
    &= \sigma(p\varphi(\sigma^{-1}\tau\sigma))\\
    &= \sigma(p)\sigma(\varphi(\sigma^{-1}\tau\sigma))\\
    &= \sigma(p)\sigma(\varphi)(\tau).
    \end{align*}
so the cocycle corresponding to $\sigma(P,p)$ is exactly the transform $\sigma(\varphi)$. It is also then clear that the transform by $\sigma \in \aut(C/A)$ depends only on its action on $B$ and therefore descends to an action of $\aut(B/A)$.

A short computation shows that if $\varphi$ and $\psi$ in $Z^{1}_{\defin}(C/B,G(C))$ are cohomologous via $g \in G(C)$, then $\sigma(\varphi)$ and $\sigma(\psi)$ in $Z^{1}_{\defin}(C/B,G(C))$ are cohomologous via $\sigma(g) \in G(C)$: Let $\varphi(\tau) = g^{-1}\psi(\tau)\tau(g)$. Then \begin{align*}
    \sigma(\varphi)(\tau) &= \sigma(\varphi(\sigma^{-1}\tau\sigma))\\&= \sigma(g^{-1}\psi(\sigma^{-1}\tau\sigma)\sigma^{-1}\tau\sigma(g)))\\
    &= \sigma(g^{-1})\sigma(\psi(\sigma^{-1}\tau\sigma)\tau\sigma(g)\\
    &= \sigma(g^{-1})\sigma(\psi)(\tau)\tau(\sigma(g)).
\end{align*}  

So $\aut(B/A)$ acts via the transform on $\HH^1_{\defin}(C/B,G(C))$ and $\mathcal{P}_{\defin}(C/B,G(C))$ equivariantly with respect to the isomorphism between them.     
\end{proof}

\begin{theorem}
    The short exact sequence $$1 \to \aut(C/B) \to \aut(C/A) \to \aut(B/A) \to 1$$ induces a short exact sequence in definable Galois cohomology $$1 \to \HH^1_{\defin}(B/A,G(B)) \to \HH^1_{\defin}(C/A,G(C)) \to \HH^1_{\defin}(C/B,G(C))^{\aut(B/A)}.$$
\end{theorem}

\begin{proof}
Let $\varphi$ be a definable cocycle from $\aut(B/A)$ to $G(B)$ with definability data $h$ and $a$. Let $\pi:\aut(C/A) \to \aut(B/A)$ be the quotient map. As in the proof of Lemma \ref{ZZ inj}, we define a map from $Z^1_{\defin}(B/A,G(B))$ to $Z^1_{\defin}(C/A,G(C))$ by sending $\varphi$ to $\varphi \circ \pi$. Now $\varphi \circ \pi$ is still a cocycle witnessed by the same definability data. Since $\pi$ is surjective, this map on cocycles is injective. Let $\varphi_1$ and $\varphi_2$ in  $Z^1_{\defin}(B/A,G(B))$ be cohomologous as witnessed by $g \in G(B)$. As $G(B) \subseteq G(C)$, the same element $g$ witnesses that $\varphi_1 \circ \pi$ and $\varphi_2 \circ \pi$ are cohomologous. So the map on cocycles descends to a well defined map on cohomology which we call $\pi^1$.\\

Let $\varphi_1$ and $\varphi_2$ be cocycles from $\aut(B/A)$ to $G(B)$ with definability data $h_1$, $\bar{a}_1$ and $h_2$, $\bar{a}_2$ respectively. Suppose that $\varphi_1 \circ \pi$ and $\varphi_2\circ\pi$ are cohomologous via $g \in G(C)$. So for all $\sigma \in \aut(C/A)$ and moreover for all $\sigma \in \aut(M/A)$, $$g^{-1}h_1(\bar{a}_1,\sigma(\bar{a}_1))\sigma(g) = h_2(\bar{a}_2,\sigma(\bar{a}_2)).$$ As $\bar{a}_1$ and $\bar{a}_2$ lie in $B$, when $\sigma \in \aut(M/B)$, $h_1(\bar{a}_1,\sigma(\bar{a}_1)) = e$ and $h_2(\bar{a}_2,\sigma(\bar{a}_2)) = e$, and then $g^{-1}\sigma(g) = e$. Thus $g$ is fixed under all automorphisms of $M$ fixing $B$. As $M$ is atomic and strongly $\omega$-homogeneous over $B$, we have $g \in dcl(B)$. As $dcl(B) = B$, we have that $g \in G(B)$. Then $g$ witnesses that $\varphi_1$ and $\varphi_2$ are cohomologous. Thus the induced map on cohomology, $\pi^1$, is injective.\\

Let $\iota:\aut(C/B) \to \aut(C/A)$ be the inclusion map. Precomposition by $\iota$ gives a map of cocycles from $Z^1_{\defin}(C/A,G(C)) \to Z^1_{\defin}(C/B,G(C))$ which induces a well defined map on cohomology $\iota^1$: The same definability data witnesses that precomposition gives a map of definable cocycles and the same element of $G(C)$ witnesses the cohomology relation on the restriction of the cocycle. Note that for cocycles in $\ZZ^1_{\defin}(C/A,G(C))$, precomposition by $\iota$ is equivalent to restriction to $\aut(C/B)$.  \\

We now show that the image of $\iota^1$ is contained in the fixed points of the action of $\aut(B/A)$ on $\HH^1_{\defin}(C/B,G(C))$ via the transform. Let $\sigma \in \aut(C/A)$ be a lift of some $\hat{\sigma} \in \aut(B/A)$. Let $\varphi \in \ZZ^1_{\defin}(C/A,G(C))$, we show that the transform $\sigma(\varphi\circ\iota)$ is cohomologous to $\varphi\circ\iota$. Let $\tau \in \aut(C/B)$. Then
\begin{align*}
    \sigma(\varphi\circ\iota)(\tau) &= \sigma(\varphi(\sigma^{-1}\tau\sigma))\\
    &= \sigma(\varphi(\sigma^{-1}))\sigma(\sigma^{-1}(\varphi(\tau)\tau(\varphi(\sigma))))\\
    &= \sigma(\varphi(\sigma^{-1})\varphi(\tau)\tau(\varphi(\sigma)))\\
    &= g^{-1}\varphi(\tau)\tau(g)
\end{align*}
where $g = \varphi(\sigma) \in G(C)$. So the cocycles $\sigma(\varphi\circ\iota)$ and $\varphi\circ\iota$ are cohomologous via $\varphi(\sigma)$. Hence the image of $\iota^1$ is contained in $\HH^1_{\defin}(C/B,G(C))^{\aut(B/A)}$. Note that we really need a class in the image: If $\varphi$ were an arbitrary cocycle in $\ZZ^1_{\defin}(C/B,G(B))$ then while it is true that $\sigma^{-1}\tau\sigma \in \aut(C/B)$, $\sigma$ is not necessarily in $\aut(C/B)$ and so we could not conclude that $\varphi(\sigma^{-1}\tau\sigma) = \varphi(\sigma^{-1})\sigma^{-1}(\varphi(\tau)\tau(\varphi(\sigma)))$ needed for the second equality of the computation.\\

Choose a cohomology class from $\HH^1_{\defin}(C/A,G(C))$ in the image of $\pi^1$. By definition this class has a representative cocycle $\varphi$ such that $\varphi = \psi \circ \pi$ where $\psi \in Z^1_{\defin}(B/A,G(B))$. Then $\varphi\circ \iota = \psi \circ \pi\circ \iota$ which is the constant cocycle. Since $\iota^1$ is a well defined map on cohomology, the image of the original class must be trivial i.e. cohomologous to the constant cocycle. So the image of $\pi^1$ is contained in the kernel of $\iota^1$. \\

We now check that the kernel of $\iota^1$ is contained in the image of $\pi^1$: Choose a cohomology class from $\HH^1_{\defin}(C/A,G(C))$ in the kernel of $\iota^1$. The first claim is that this class must contain a representative mapping to the constant cocycle: If $\varphi'$ is a representative of a class in the kernel of $\iota^1$ then there is a $g \in G(C)$ such that $\forall \sigma \in \aut(C/B)$, $\varphi'(\sigma) = g^{-1}\sigma(g)$. But then $g\varphi'\sigma(g^{-1})$ is another representative of the class of $\varphi'$ which maps to the constant cocycle.\\

Let $\varphi$ be a representative of our cohomology class in the kernel of $\iota^1$ such that $\varphi \circ \iota$ is the constant cocycle. The cocycle $\varphi$ is a map from $\aut(C/A)$ to $G(C)$ which sends the subgroup $\aut(C/B)$ to the identity of $G(C)$.\\

We now show that the definability data for $\varphi$ can be re-chosen so that the parameter comes from $B$: Let $h(\bar{x},\bar{y})$, $\bar{a}$ be the original definability data for $\varphi$. Let $Z$ be the $A$-definable set $\tp(\bar{a}/A)$; recall that $M$ is isolated over $A$. From $h(\bar{x},\bar{y})$ and $\bar{a}$, we can produce, as in Lemma \ref{cocycle to pPHS}, a pointed PHS for $G$ whose underlying definable set is $(Z\times G)/E$ where $E$ is the $A$-definable equivalence relation defined by $(\bar{a},g)E(\bar{a}',g')$ if $h(\bar{a},\bar{a}')g' = g$. The basepoint of this PHS is the $E$-class $(\bar{a},e)/E$ which furnishes new definability data for $\varphi$ as in Lemma \ref{pPHS to cocycle}.

Now, let $\sigma \in \aut(C/B)$. Then $\sigma((\bar{a},e)/E) = (\bar{a},e)/E$ as $h(\bar{a},\sigma(\bar{a}))e = e$. So $(\bar{a},e)/E$ is fixed pointwise by the action of $\aut(C/B)$, so $(\bar{a},e)/E \in B$.\\

We may now assume that $\varphi$ has definability parameter $\bar{a}$ coming from $B$. As in the proof of Lemma \ref{ZZ inj}, the definability condition $\varphi(\sigma) = h(\bar{a},\sigma(\bar{a}))$ shows $\varphi(\sigma) \in \dcl(\bar{a},\sigma(\bar{a}))$ and so by the normality of $B$ and the fact that $\dcl(B) = B$, the image of $\varphi$ is contained in $G(B)$. As the definability parameter is contained in $B$, $\varphi$ descends to a well defined map $\tilde{\varphi}$ from $\aut(B/A)$ to $G(B)$, and moreover $\tilde{\varphi}\circ\pi = \varphi$. We show that $\tilde{\varphi}$ still satisfies the cocycle condition: Let $\sigma,\tau \in \aut(C/A)$ and $\tilde{\sigma},\tilde{\tau} \in \aut(B/A)$ with $\pi(\sigma) = \tilde{\sigma}$ and $\pi(\tau) = \tilde{\tau}$ 
\begin{align*}
    \tilde{\varphi}(\tilde{\sigma}\tilde{\tau}) &= \varphi(\sigma\tau)\\
    &= \varphi(\sigma)\sigma(\varphi(\tau))\\
    &= \tilde{\varphi}(\tilde{\sigma})\tilde{\sigma}(\tilde{\varphi}(\tilde{\tau}))
\end{align*}
where the last equality holds because the image of $\varphi$ is contained in $B$.  The same definability data for $\varphi$ shows that $\tilde{\varphi}$ is a definable cocycle. The class of $\varphi$ is then the image of the class of $\tilde{\varphi}$ under $\pi^1$.

\end{proof}

%\section{Consequences of the short exact sequence}

%Let $M = acl(A)$. Then $$\aut(M/A) \cong \varprojlim \aut(B/A)$$ where the limit ranges over the inverse system of automorphism groups of finitely generated (in the sense of $\dcl^{eq}$) normal extensions $B$ of $A$. For any $G$ definable over $A$, the usual colimit formula holds $$\HH^1_{\defin}(M/A,G(M)) \cong \varinjlim \HH^1_{\defin}(B/A,G(B)).$$

%In the setting of an arbitrary $\omega$-stable theory $T$ the situation is more complicated.\\

%Let $G$ be an $A$-definable group such that $G(A) = G(M)$. Let $X$ be an $A$-definable PHS for $G$. Let $b \in X(M)$. Then $\tp(b/A)$ is internal to $G(A)$ and weakly orthogonal to $A$. The extension $A \subseteq \dcl^{eq}(A,b)$ is a definable Galois extension with Galois group definably embedding in $G$. It follows that  $$\HH^1_{\defin}(M/A,G(M)) \cong \varinjlim \HH^1_{\defin}(B/A,G(B))$$ where the limit is taken over $G$-strongly normal extensions of $A$.\\

%\begin{lemma}
%    An $A$-definable cocycle $\varphi:\aut(B/A) \to G(B)$ gives rise to a pointed PHS for $G(B)$.
%\end{lemma}

%\begin{proof}
%    This can be summarized briefly. Let $\pi:\aut(M/A) \to \aut(B/A)$ be the surjective restriction homomorphism. Then $\varphi\circ \pi$ is a definable cocycle from $\aut(M/A) \to G$ which under the correspondence of \cite{pillay_1997} defines a pointed PHS for $G$ with a $B$-point. As in remark \ref{Bpts of PHS} the $B$-points of this PHS are then a PHS for $G(B)$.

%    \textit{This will show that the first square in the chain map commutes.}
%\end{proof}

\newpage
\bibliography{main.bib}{}
\bibliographystyle{plain}
\nocite{*}

\end{document}